\documentclass[11pt]{amsart}
\usepackage{amsfonts}
\usepackage{amsmath}
\usepackage{amssymb}
\usepackage{color}
\usepackage[all]{xy}
\usepackage{graphicx}
\usepackage{xspace}
\usepackage{wrapfig}
\usepackage{soul}
\usepackage{cancel}
\usepackage{ulem}
\usepackage{axodraw}

\newtheorem{theorem}{Theorem}

\newtheorem{corollary}[theorem]{Corollary}

\newtheorem{definition}[theorem]{Definition}

\newtheorem{lemma}[theorem]{Lemma}

\newtheorem{proposition}[theorem]{Proposition}
\newtheorem{remark}[theorem]{Remark}

\newcommand{\un}{\hbox{\bf 1}}
\newcommand{\id}{id}

\def\racine{{\scalebox{0.3}{
\begin{picture}(12,12)(38,-38)
\SetWidth{0.5} \SetColor{Black} \Vertex(45,-30){6}
\end{picture}
}}}
  
 \def\arbrea{\,{\scalebox{0.15}{ 
  \begin{picture}(8,55) (370,-248)
    \SetWidth{2}
    \SetColor{Black}
    \Line(374,-244)(374,-200)
    \Vertex(374,-197){9}
    \Vertex(375,-245){12}
  \end{picture}
}}\,}

 \def\arbreba{\,{\scalebox{0.15}{ 
\begin{picture}(8,106) (370,-197)
    \SetWidth{2}
    \SetColor{Black}
    \Line(374,-193)(374,-149)
    \Vertex(374,-146){9}
    \Vertex(375,-194){12}
    \Line(374,-142)(374,-98)
    \Vertex(374,-95){9}
  \end{picture}
}}\,}

 \def\arbrebb{\,{\scalebox{0.15}{ 
  \begin{picture}(48,48) (349,-255)
    \SetWidth{2}
    \SetColor{Black}
    \Vertex(375,-252){12}
    \Line(376,-250)(395,-215)
    \Line(373,-251)(354,-214)
    \Vertex(353,-211){9}
    \Vertex(395,-213){9}
  \end{picture}
}}}

\def\arbreca{\,{\scalebox{0.15}{
\begin{picture}(8,156) (370,-147)
    \SetWidth{2}
    \SetColor{Black}
    \Line(374,-143)(374,-99)
    \Vertex(374,-96){9}
    \Vertex(375,-144){12}
    \Line(374,-92)(374,-48)
    \Vertex(374,-45){9}
    \Line(374,-42)(374,2)
    \Vertex(374,5){9}
  \end{picture}
}}\,}

\def\arbrecb{\,{\scalebox{0.15}{
\begin{picture}(48,94) (349,-255)
\SetWidth{2}
\SetColor{Black}
\Line(376,-204)(395,-169)
\Line(373,-205)(354,-168)
\Vertex(353,-165){9}
\Vertex(395,-167){9}
\Vertex(374,-205){9}
\Line(374,-246)(374,-209)
\Vertex(374,-252){12}
\end{picture}}}\,}

\def\arbrecc{\,{\scalebox{0.15}{
 \begin{picture}(48,98) (349,-205)
    \SetWidth{2}
    \SetColor{Black}
    \Vertex(375,-202){12}
    \Line(376,-200)(395,-165)
    \Line(373,-201)(354,-164)
    \Vertex(353,-161){9}
    \Vertex(395,-163){9}
    \Line(353,-160)(353,-113)
    \Vertex(353,-111){9}
  \end{picture}
}}\,}

\begin{document}


\title[${\mathsf{T}}$-ordering and pre-Lie algebras]{The pre-Lie structure of the time-ordered exponential} 


\author{Kurusch Ebrahimi-Fard}
\address{ICMAT,
		C/Nicol\'as Cabrera, no.~13-15, 28049 Madrid, Spain.
		On leave from UHA, Mulhouse, France.}
         \email{kurusch@icmat.es, kurusch.ebrahimi-fard@uha.fr}         
         \urladdr{www.icmat.es/kurusch}

\author{Fr\'ed\'eric Patras}
\address{Univ.~de Nice,
Labo.~J.-A.~Dieudonn\'e,
         		UMR 7351, CNRS,
         		Parc Valrose,
         		06108 Nice Cedex 02, France.}
\email{patras@math.unice.fr}
\urladdr{www-math.unice.fr/$\sim$patras}

\date{May 16, 2013}


\begin{abstract}
The usual time-ordering operation and the corresponding time-ordered exponential play a fundamental role in physics and applied mathematics. In this work we study a new approach to the understanding of time-ordering relying on recent progress made in the context of enveloping algebras of pre-Lie algebras. Various general formulas for pre-Lie and Rota--Baxter algebras are obtained in the process. Among others, we recover the noncommutative analog of the classical Bohnenblust--Spitzer formula, and get explicit formulae for operator products of time-ordered exponentials.
\end{abstract}


\maketitle
\tableofcontents


\section*{Introduction}
\label{sect:intro}

Time-ordering appears naturally in the resolution of matrix or operator linear differential equations, through the so-called time-ordered exponential (aka Picard, Chen or Dyson series). The algebraic structure of the time-ordered exponential has been studied intensively. A classical example is provided by the solution to the Baker--Campbell--Hausdorff problem (of computing the logarithm of the solution of a linear differential equation or, almost equivalently, the logarithm of a product of exponentials). 

In the present article, we address the problem of studying time-ordered exponentials and their natural generalizations to arbitrary Rota--Baxter algebras from a renewed point of view. Namely, we relate these structures to the one of enveloping algebras of pre-Lie algebras, taking advantage of recent advances made in the subject. This approach sheds new light on classical topics (such as the aforementioned Baker--Campbell--Hausdorff problem) but, above all, it leads to the derivation of new formulas as well as structural results for pre-Lie and Rota--Baxter algebras. In our opinion, it gives rise to changes in the conceptual understanding of associative Rota--Baxter algebras and their Lie--theoretical properties.

Let us recall briefly the background of these ideas. Pre-Lie algebras first appeared in the works of Gerstenhaber (in deformation theory) and of Vinberg (in differential geometry) in the early 1960s. Later they occurred also in mathematical control theory under the name of chronological algebras \cite{AG}. See \cite{Cartier2,Manchon} for concise reviews. After the seminal works by Connes and Kreimer \cite{CKII} in perturbative quantum field theory, and soon after by Chapoton and Livernet \cite{ChaLiv} in universal algebra and the theory of operads, it became clear that the notion of pre-Lie algebra plays also an important role in combinatorial algebra. 

Pre-Lie algebras are Lie admissible, that is, the antisymmetrized pre-Lie product defines a Lie bracket. The fruitful interplay between algebra and geometry, which is reflected in the notion of pre-Lie algebra becomes visible in the context  Butcher's theory of numerical integration methods \cite{HWL}. Here, based on Cayley's historical 1857 paper ``On the theory of analytic forms called trees'', rooted trees provide the basic entities. They encode algebraically the notion of flat and torsion free connections, i.e., the differential geometry of euclidean spaces.

The Butcher--Connes--Kreimer Hopf algebra of rooted trees is the dual of the Grossman--Larson Hopf algebra described in \cite{GL}. The elements of the Grossman--Larson algebra include the so-called Butcher group. It has received considerable attention during recent years, and provides the basis for a modern understanding of the pre-Lie structure underlying the algebra of differential operators. Guin and Oudom generalized in \cite{OudGui} the Grossman--Larson construction by showing that the symmetric algebra of any pre-Lie algebra can be equipped with an associative product giving rise to a Hopf algebra, which is isomorphic to the enveloping algebra of the Lie algebra coming from the pre-Lie algebra.         

In the present work we would like to put these findings to use in the context of associative Rota--Baxter algebras. Such algebras are defined in terms of linear maps that satisfy a modified integration by parts identity, known as Rota--Baxter relation. They naturally come equipped with a second associative product as well as a natural pre-Lie algebra structure. The latter results from the combination of integration by parts, or its generalization in terms of the Rota--Baxter identity, with the Jacobi identity. 

Following Rota \cite{Rota2,Rota3,Rota4}, in recent years the theory of Rota--Baxter algebras of arbitrary weight \cite{Baxter} has emerged as an appropriate setting for studying algebraic and combinatorial aspects underlying the notion of integration. See \cite{EP} for a review. Let us mention that in particular we are interested in a more detailed understanding of the classical notion of time-ordered exponential from an algebraic point of view. We emphasize that our results provide a refined and more conceptional understanding of the main part of \cite{EMP}, i.e., the noncommutative generalization of the classical Bohnenblust--Spitzer identity. This generalization follows from a non-trivial interplay between the aforementioned associative and pre-Lie structures on Rota--Baxter algebras. It matches Guin's and Oudom's construction of a Grossman--Larson type algebra out of a pre-Lie algebra. 

\smallskip

The paper is organized as follows. In Section \ref{sect:GLA} we recall the definition of the Grossman--Larson algebra and prove the key identity in Theorem \ref{FundThm}, relating commutative and (Grossman--Larson's) associative products of trees. Section \ref{sect:UPL} explains briefly how these results extend to arbitrary enveloping algebras of pre-Lie algebras. Closed formulas for the brace operations arising naturally from pre-Lie products are presented in Corollary \ref{cor:closed}. In the next section we describe explicitly the canonical embedding of the Grossman--Larson algebra into  the free Rota--Baxter algebra in Theorem \ref{thm:keyth}. From this we recover the noncommutative Bohnenblust--Spitzer formula. Section \ref{sect:ToBSp} investigates the fine combinatorial and pre-Lie structures of the time-ordered exponential. Section \ref{sect:Magnus} explains briefly how the links between pre-Lie and Rota--Baxter algebras shed new light on the continuous Baker--Campbell--Hausdorff problem and the related combinatorics of symmetric groups. In Section \ref{sect:prodToExp} we show how the product of time-ordered exponentials can be rewritten in the Grossman--Larson algebra as a proper time-ordered exponential in terms of symmetric braces. The classical BCH-formula in the context of pre-Lie algebra is recovered in this setting. 

\medskip     
 
\noindent In the following we fix $k$ to be a field of characteristic zero over which all algebraic structures are defined.

\vspace{0.4cm}

\noindent {\bf{Acknowledgements}}: The first author is supported by a Ram\'on y Cajal research grant from the Spanish government, as well as the project MTM2011-23050 of the Ministerio de Econom\'{i}a y Competitividad. The second author acknowledges support from the grant ANR-12-BS01-0017, Combinatoire Alg\'ebrique, R\'esurgence, Moules et Applications. Both authors were supported by the CNRS GDR Renormalisation.


\section{The Grossman--Larson algebra}
\label{sect:GLA}

Recall that a tree $t$ is a finite, non empty, connected and simply connect graph with a distinguished vertex called the root. 
Here $T$ denotes the set of rooted trees: 
$$
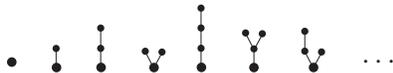
\racine\quad
\arbrea \quad
\arbreba \quad
\arbrebb \quad
\arbreca \quad 
\arbrecb \quad
\arbrecc \quad  
\cdots
$$
Let ${\mathcal D}$ be an arbitrary set. By a (${\mathcal D}$-)decorated tree we mean that each vertex (including the root vertex) of a tree carries a decoration by an element of ${\mathcal D}$; equivalently a map is given from the set of vertices of the tree $t \in T$ to ${\mathcal D}$. From now on, all trees (and related objects) will carry decorations (we will therefore omit to mention that they are ${\mathcal D}$-decorated).

The grafting of a tree $t'$ on another tree $t$ results in an element of the linear span $\mathcal{T}$ of the set $T$ of trees. It is the sum written $t \curvearrowleft t'$ of all the trees obtained by grafting (i.e.~creating an edge linking) the root of $t'$ with an arbitrary vertex of $t$. A forest $F=t_1 \cdots t_n$ is a commutative product of trees $t_i$ (or equivalently, an element of the free commutative semigroup generated by the trees). The weight of a forest, $w(F)$, is the number of trees in the forest, i.e., for $F=t_1 \cdots t_n$, $w(F)=n$.

The linear span of the set of forests, written $GL^+$, is naturally equipped with two products. The first one, written $\circ$ (which we omit, as it is usually done in the case of polynomial algebras) is associative and commutative. It is simply the bilinear extension of the commutative product of forests, i.e., for $F=t_1 \cdots t_n$, $F' = t_1' \cdots t_m'$, 
$$
	F \circ F'=FF':=t_1 \cdots t_n t_1'\cdots t_m'.
$$
Note that $w(FF')=n+m$. The second one, called Grossman--Larson product, and written $\ast$, is associative, but not commutative. It is defined as follows:
\begin{equation}
\label{gsProduct}
	(t_1 \cdots t_n) \ast (t_1' \cdots t_m') = \sum\limits_f F_0(t_1\curvearrowleft F_1) \cdots (t_n\curvearrowleft F_n),
\end{equation}
where the sum is over all functions $f$ from $\{1,\ldots ,m\}$ to $\{0,\ldots,n\}$ and $F_i:=\prod_{j\in f^{-1}(i)}t_j'$. Notice that, in this definition, we used the fact that $GL^+$ is also equipped with a natural action on the linear span of trees extending the product $\curvearrowleft$. This action can be defined recursively, together with the $\ast$ product using the identity: 
\begin{equation}
\label{envact}
	t\curvearrowleft (F\ast G):=(t\curvearrowleft F)\curvearrowleft G,
\end{equation}
where $t$ is any tree and $F,G$ are arbitrary forests. As an example:
$$
	t \curvearrowleft (t' t''):=(t\curvearrowleft t')\curvearrowleft t'' - t\curvearrowleft (t'\curvearrowleft t'').
$$
Conceptually, this follows from the fact that the product $\ast$ equips $GL^+$ with the structure of an enveloping algebra over $\mathcal{T}$, equipped with the Lie bracket $[t,t']:=t\curvearrowleft t'-t'\curvearrowleft t$ --we postpone explanations and refer to the next section for details. 

Let us mention that the coproduct dual to the product $\ast$ is easy to describe graphically in terms of cuts on the branches of trees: it is the ``sum over all admissible cuts'' in the terminology of Connes and Kreimer; we refer the interested reader to \cite{CKII}  for further details. The algebra $(GL,\ast)$, named Grossman--Larson algebra, is the unital algebra obtained by adjoining a unit to $GL^+$. It is sometimes convenient to treat this unit as a tree -- which is then referred to as the empty tree and denoted by $e$.

Following \cite{OudGui} we can further extend the aforementioned action of forest on trees to all of  $GL$ in terms of the following

\begin{definition} \label{def:forest}
Let $F,G,H \in GL$, then we define $F \curvearrowleft e := F$, and:
\begin{eqnarray}
\label{forestaction}	
	FG \curvearrowleft H &:=& (F \curvearrowleft H_1)(G\curvearrowleft H_2).
\end{eqnarray}
\end{definition} 

We used Sweedler's notation for the coproduct defined naturally on the forest $H \in GL$, $\Delta(H):=H_1 \otimes H_2$, which is simply deshuffling, i.e., trees in $T$ are defined to be primitive elements. For details we refer the reader to \cite{OudGui}.    

Let us return to (\ref{gsProduct}). Direct inspection shows that, for forests $F$ and $F'$ as above, $F \ast F' = FF' + R$, where the reminder is a sum of forests of weight strictly less than $w(FF')=n+m$. Let us order arbitrarily the set of trees. It follows then, by the usual triangularity argument used to construct the Poincar\'e--Birkhoff--Witt (PBW) basis of enveloping algebras (see, e.g., \cite{reutenauer}), that the set of products $t_1\ast \cdots \ast t_n$, where $t_1\leq \cdots \leq t_n$, forms a basis of $GL$.

In this section, we aim at expanding forests as linear combinations of iterated Grossman--Larson products of trees. Besides being of general interest with respect to the $GL$ algebra, the motivation for such a computation comes from the theory of Rota--Baxter algebras and related identities, such as the Bohnenblust--Spitzer formula \cite{EMP}. This will be explained further below in later sections of the present article.

\begin{theorem}\label{FundThm}
For $t_1,\ldots ,t_n$ trees in $GL$, we have:
 \begin{equation}
 \label{Fondam}
	t_1\cdots t_n = \sum\limits_{P_1,\ldots ,P_k} (-1)^{n-k} t_{P_1}\ast \cdots \ast t_{P_k},  
 \end{equation}
where $P:=\{P_1, \ldots,P_k\}$ runs over the partitions of $[n]=\{1,\ldots, n\}$, i.e., $P_1\coprod \cdots \coprod P_k=[n]$, such that $\sup(P_i)>j$  $\forall j \in P_k$, $k<i$. Besides, for $P_i=\{p_1^i,\ldots ,p_h^i\}$, where the $p_j^i$ are in natural order, we set:
$$
	t_{P_i}:=\sum_{\sigma\in S_{h-1}} t_{p_{\sigma(1)}^i} \curvearrowleft (t_{p_{\sigma(2)}^i} \curvearrowleft 
	( \cdots \curvearrowleft (t_{p_{\sigma(h-1)}^i}\curvearrowleft t_{p_{h}^i})\cdots ).
$$
\end{theorem}

\begin{proof}
The proof follows by induction of $n$. Notice first that:
$$
	t_1\cdots t_n = (t_1\cdots t_{n-1})\circ t_n 
	= (t_1\cdots t_{n-1}) \ast t_n - \sum\limits_{i=1}^{n-1}t_1\cdots t_{i-1}(t_i \curvearrowleft t_n)\cdots t_{n-1}.
$$
The first term on the right hand side corresponds (by induction) to the terms such that the block $P_k=\{n\}$ in the above formula. The other terms can be rewritten:
$$
	t_1\cdots t_{i-1}(t_i\curvearrowleft t_n)\cdots t_{n-1}
	=t_1\cdots t_{n-1}(t_i\curvearrowleft t_n)=: U_1 \cdots U_{n-1}
$$
(where $U_1=t_1,\ldots ,U_{i-1}=t_{n-1},U_i=t_{i+1},\ldots ,U_{n-1}=t_i \curvearrowleft t_n$). The statement follows by substituting these values of the $U_i$s in the expansion of $U_1\cdots U_{n-1}$.
\end{proof}

Two remarks are in order. First, note that since the left hand side of the formula is $S_n$-invariant, other expansions could be obtained by using different parametrizations of the products (e.g., rewriting $t_1\cdots t_n=t_n \cdots t_1$, and starting the recursion with respect to $t_1$ would produce a different expansion). Other, less elegant, expansions can also be obtained by using the rewriting $t_1\cdots t_{i-1}(t_i\curvearrowleft t_n)\cdots t_{n-1}=:U_1\cdots U_{n-1}$ instead of the one we used.  Second, identity (\ref{Fondam}) is somehow surprising, since the left had side is evidently invariant by permutations of the trees, whereas this is much less obvious for the expression on the right hand side.


\section{Enveloping algebras of pre-Lie algebras}
\label{sect:UPL}

The previous theorem applies to two different situations, which are equally interesting on their own, since each of them encompasses important application domains, i.e., pre-Lie algebras and Rota--Baxter algebras. Let us consider first pre-Lie algebras and their enveloping algebras.

Recall that a right pre-Lie algebra is a vector space $L$ endowed with a bilinear product $\curvearrowleft$ : $L \times L \longrightarrow L$ satisfying the relation: 
$$
	(x \curvearrowleft y) \curvearrowleft z - x \curvearrowleft (y \curvearrowleft z) 
	= (x \curvearrowleft z) \curvearrowleft y - x \curvearrowleft (z \curvearrowleft y),\qquad \forall x, y, z \in L.
$$
The case of left pre-Lie algebra follows from exchanging $x$ and $y$ instead of $y,z$. Note that $L$ is Lie admissible, i.e., the bracket defined by antisymmetrization:
$$
	[a, b] := a \curvearrowleft b - b \curvearrowleft a,\qquad \forall a, b \in L
$$
endows $L$ with the structure of Lie algebra. Note that a commutative pre-Lie algebra is associative.

The $\curvearrowleft$ product of trees introduced in the last section happens to define the structure of a free pre-Lie algebra on the linear span of trees. This was proven by Chapoton and Livernet in the seminal work \cite{ChaLiv2}. It can also be deduced from the description of free pre-Lie algebras by Agrachev and Gamkrelidze in \cite{AG}. As a vector space, $GL$ is isomorphic to the symmetric algebra over $\mathcal T$. It happens so that the product $\ast$ equips $GL$ with the structure of an enveloping algebra over $\mathcal T$. This phenomenon is specific to pre-Lie algebras. Indeed, usually, the Poincar\'e--Birkhoff--Witt theorem states that the symmetric algebra $S(L)$ over a Lie algebra $L$ is naturally isomorphic to $U(L)$, the enveloping algebra of $L$, but there is no simple formula allowing to transfer the associative product on $U(L)$ to $S(L)$. As a corollary of this interpretation of $GL$ as an enveloping algebra, it follows that there is a natural action of $GL$ on $\mathcal T$ extending the pre-Lie product. This is nothing but the action described by equation (\ref{envact}) in the previous section. Note that this action is also written using the brace notation, $\{l;l_1\cdots l_n\}:=l\curvearrowleft (l_1\cdots l_n)$, or $l\{l_1\cdots l_n\}$ (where $l$ as well as the $l_i$ belong to $\mathcal T$). This is naturally motivated by (\ref{envact}) and the enveloping algebra structure of $GL$, which yields:
\begin{eqnarray*}	
	\{\{l;l_1 \cdots l_n\};{p_1 \cdots p_m}\}
	&=&(l\curvearrowleft (l_1 \cdots l_n))\curvearrowleft (p_1 \cdots p_m)\\
	&=&l\curvearrowleft (l_1 \cdots l_n\ast p_1 \cdots p_m)
	=\sum\limits_{f}l\curvearrowleft (P_0(l_1\curvearrowleft P_1) \cdots (l_n\curvearrowleft P_n)),
\end{eqnarray*}	
with the (self-explaining adaptation of the) notation of equation (\ref{envact}). So that finally we obtain:
$$
	\{\{l;l_1 \cdots l_n\};{p_1 \cdots p_m}\}=\sum\limits_{f}\{l;P_0(l_1\curvearrowleft P_1) \cdots (l_n\curvearrowleft P_n)\},
$$
the defining relation of symmetric brace algebras as stated in \cite{OudGui} (see also e.g.~\cite{lada}). We refer to \cite{OudGui} for further insights and details on the subject. 

Actually, since $\mathcal T$ is free as a pre-Lie algebra, the construction of the Grossman--Larson algebra can be extended to construct the enveloping algebra of $L$ for an arbitrary pre-Lie algebra $L$. That is, the algebra $S(L)$ of polynomials over $L$ can be equipped with an associative product $\ast$ making $(S(L),\ast)$ an enveloping algebra. The product law is given by formula (\ref{gsProduct}). See \cite{OudGui} for details.

Since the pre-Lie algebra of trees is free, it also follows that formulae in the $GL$ algebra are universal: they hold for the enveloping algebra of an arbitrary pre-Lie algebra. In particular, formula (\ref{Fondam}) holds in this setting.

\begin{proposition}
\label{prop:Envel}
For an arbitrary pre-Lie algebra $L$ and arbitrary elements $l_1,\ldots ,l_n$ of $L$, we have (in the Guin--Oudom presentation of the enveloping algebra of $L$ as $S(L)$ equipped with the $\ast$ product) the identity:
$$ 
	l_1\cdots l_n = \sum\limits_{P_1,\ldots ,P_k} (-1)^{n-k}l_{P_1}\ast \cdots \ast l_{P_k},  
$$
with the same notation for the $P_i$ and the $l_{P_i}$ as in Theorem~\ref{FundThm}.
\end{proposition}

\begin{corollary}\label{cor:closed}
In particular, we get for the symmetric braces $\{l;l_1\cdots l_n\}$ the expansion:
$$
	\{l;l_1\cdots l_n\} = \sum_{P_1,\ldots ,P_k}(-1)^{n-k}
	(\cdots ((l\curvearrowleft l_{P_1})\curvearrowleft l_{P_2}) \cdots \curvearrowleft l_{P_k}).
$$
\end{corollary}

The notation is as above. In particular, the sum runs over partitions of $[n]$ obeying the particular statistics of Theorem~\ref{FundThm}. Indeed, we have:
\begin{eqnarray*}
	\{l;l_1\cdots l_n\} &=& \sum_{P_1,\ldots ,P_k}(-1)^{n-k}l\curvearrowleft (l_{P_1}\ast \cdots \ast l_{P_k})\\
			   &=& \sum_{P_1,\ldots ,P_k}(-1)^{n-k}(\cdots ((l\curvearrowleft l_{P_1})\curvearrowleft l_{P_2}) 
			   		\cdots \curvearrowleft l_{P_k}),
\end{eqnarray*}


\section{Grossmann--Larson and Rota--Baxter algebras}
\label{sect:RBA}

First, this section recalls briefly the notion of Rota--Baxter algebra. A particularly interesting topic from this perspective is the noncommutative Bohnenblust--Spitzer identity of \cite{EMP}. The commutative version of this identity plays a key role in Rota's approach to (Rota--)Baxter algebras by means of symmetric functions (see, e.g., the introduction \cite{Rota1}). Similarly, the identity is central in the correspondence we establish between enveloping algebras of pre-Lie algebras and Rota--Baxter algebras. In particular, we will show how the noncommutative Bohnenblust--Spitzer identity follows naturally from our previous results.

\smallskip

Recall the definition of a unital Rota--Baxter algebra of weight $\theta \in k$ \cite{Baxter,EMP,EP,Rota1,Rota2}. It is an associative algebra $A$ with unit $\un$, equipped with a linear endomorphism such that for all $x,y \in A$: 
\begin{equation}
\label{RBeq}
	R(x)R(y) = R(R(x)y + xR(y)) + \theta R(xy).
\end{equation}
Note that the map $\tilde{R}:=-\theta \id - R$ as well satisfies (\ref{RBeq}), and that both images $R(A)$ and $\tilde{R}(A)$ are subalgebras of $A$. Natural examples of weights zero and non-zero are the indefinite Riemann integral, denoted $I$, and its corresponding discrete Riemann summation operators, respectively. Another class of examples are orthogonal projectors $\pi_{\mp}$ to subalgebras $A_-=\pi_-(A)$ respectively $A_+=(\id-\pi_-)(A)$, corresponding to a direct decomposition $A=A_+ \oplus A_-$. Note that in the case of the weight zero Riemann integral $\tilde{I}=-I$. We write $RB$ for the free Rota--Baxter algebra over a set of generators $\mathcal D$. Constructions of free Rota--Baxter algebra have been given in \cite{AgMo,EG,EGBP,EMP}. Baxter, Rota and Cartier described the free commutative Rota--Baxter algebra much earlier \cite{Cartier,LG,Rota1,Rota2}. Especially Cartier contributed to its understanding by using quasi-shuffle like products \cite{Hoffman} in the construction of free commutative $RB$. We refer to \cite{EP,Rota2,Rota3} for more details. 

The following simple observations play a crucial role in applications of Rota--Baxter algebras.

\begin{lemma} 
\label{lem:pre-LieRB}
Let $(A,R)$ be an associative Rota--Baxter algebra of weight $\theta$. The two binary operations:
 \allowdisplaybreaks{
\begin{eqnarray}
	a \triangleleft_\theta b 	&:=& aR(b) - R(b)a + \theta ab = [a,R(b)] + \theta ab, 		\label{lem:rightRBpre-Lie}\\
 	a \triangleright_\theta b 	&:=& R(a)b - bR(a) - \theta ba = [R(a),b] - \theta ba 	 	\label{lem:leftRBpre-Lie}
\end{eqnarray}}
define right respectively left pre-Lie structures on $A$.
\end{lemma}
Note that $a \triangleleft_\theta b = - b \triangleright_\theta a$. For $\theta =0$, we obtain $a \triangleright_0 b = [R(a),b]=ad_{R(a)}(b)$. We denote the pre-Lie algebra corresponding to (\ref{lem:leftRBpre-Lie}) by $P_A$. Beside the pre-Lie products, one can define yet another product on Rota--Baxter algebras (the so-called double product).

\begin{lemma}
$A$ equipped with the ``double'' product:
\begin{equation}
\label{def:RBdouble}
        a \ast_\theta b := R(a)b + aR(b) + \theta ab
\end{equation}
is again a Rota--Baxter algebra of weight $\theta$ with Rota--Baxter map $R$. We denote it by $(A_\theta,R)$.
\end{lemma}

It is easy to see that $R(a \ast_\theta b) = R(a)R(b)$. Both the pre-Lie products and the new associative product give rise to Lie algebras.  The following result is important in view of our forthcoming developments: it allows to relate the pre-Lie and the double product on a Rota--Baxter algebra to the pre-Lie and associative products in the theory of enveloping algebras of pre-Lie algebras.

\begin{lemma}
\label{lem:GL2RB}
The two Lie brackets associated respectively to the products $ \triangleright_\theta$ and $\ast_\theta$ coincide. In particular, since $GL$ is the enveloping algebra of the free pre-Lie algebra over $\mathcal D$, the embedding of $\mathcal D$ into $RB$ induces a map $\iota$ from $GL$ to $RB$. This map, which sends $(\curvearrowleft , \ast)$ to $(\triangleleft_\theta, \ast_\theta)$, is itself an embedding.
\end{lemma}

\begin{proof}
The lemma is a consequence of the construction of free Rota--Baxter algebras \cite{AgMo,EG} and of Chapoton's works on dendriform enveloping algebras of brace algebras \cite{chapo}. We review briefly the chain of arguments allowing to deduce the Lemma from these works.

Recall first that a noncommutative shuffle algebra, or dendriform algebra, is an algebra equipped with two products, $\downarrow$ and $\uparrow$, satisfying the noncommutative version of the Eilenberg--MacLane--Sch\"utzenberger axioms of shuffle algebras (see e.g., \cite{EP}), that is:
\begin{eqnarray}
\label{demishuffleNC}
	(a\uparrow b)\uparrow c		&=&a\uparrow (b\uparrow c+b\downarrow c)\\
	a\downarrow (b\uparrow c)	&=&(a\downarrow b)\uparrow c\\
	a\downarrow (b\downarrow c)	&=&(a\uparrow b+a\downarrow b)\downarrow c.
\end{eqnarray}
These axioms are satisfied for example by the decomposition of the shuffle product acting of simplicial chain complexes in algebraic topology into two half-shuffles (they actually underly the celebrated Eilenberg--MacLane proof of the associativity of shuffle products in topology, \cite{EM1,EM2}). Another fundamental example is the Malvenuto--Reutenauer algebra, or algebra of Free Quasi-Symmetric Functions. Its dendriform structure has been studied intensively by Foissy, see e.g.~\cite{Foissy}.

There is a forgetful functor from Rota--Baxter algebras to dendriform algebras defined in terms of products by: 
$$
	x\uparrow y:= xR(y)+\theta xy,\ x\downarrow y:=R(x)y,
$$
such that $\uparrow + \downarrow=\ast_\theta$. The construction of free Rota--Baxter algebras shows that the embedding of $\mathcal D$ into $RB$  induces an embedding of the free dendriform algebra over $\mathcal D$ into $RB$, see \cite{AgMo}. The Lemma follows therefore if the free pre-Lie algebra and its enveloping algebra embed into the free dendriform algebra (over the same sets of generators). This last result is a consequence of \cite{chapo}: the free pre-Lie algebra over $\mathcal D$ embeds (as a subobject of the free brace algebra over $\mathcal D$ -- the fact that there exists a canonical map from the free pre-Lie algebra to the free brace algebra is also a direct consequence of Chapoton's work, see also \cite{Foissy2}) into the free dendriform algebra. 
\end{proof}

At this point we would like to introduce some notation. We define $R^{[n}(b_1,\ldots ,b_n)$ inductively by $R^{[1}(b_1):=b_1$, and
 \begin{equation}
 \label{notationR}
	R^{[n}(b_1,\ldots ,b_n):= R(R^{[n-1}(b_1,\ldots ,b_{n-1})) b_n.
 \end{equation}
We denote $R^{[n}(b,\ldots ,b)$ by $R^{[n}(b)$. The element $R^{[n]}(b_1,\ldots ,b_n)$ is set to be $R(R^{[n}(b_1,\ldots ,b_n))$, and $R^{[n]}(b,\ldots ,b)$ is just denoted $R^{[n]}(b)$. 

\begin{remark}{\rm{1) In \cite{EGBP,EMP} the notions of Spitzer algebra $\mathcal{S}$ and double Spitzer algebra $\mathcal C$ were introduced. The former is a graded connected cocommutative Hopf algebra freely generated by the elements $R^{[n]}(x)$, $n>0$. The latter is freely generated with respect to the product (\ref{def:RBdouble}) by the elements $R^{[n}(x)$, $n>0$, and is isomorphic as a Hopf algebra to $\mathcal S$.

2) By requiring the free generators to form a sequence of divided powers, $\mathcal{S}$ can be shown to be naturally isomorphic to the Hopf algebra of noncommutative symmetric functions \cite{gelfand} or, equivalently, to the descent algebra. The embedding in Lemma \ref{lem:GL2RB}, which is an embedding of enveloping algebras, is compatible with the Hopf algebra structures (enveloping algebras are Hopf algebras), in particular with the one on the Spitzer algebra.

3) An important property of the double Spitzer algebra $\mathcal C$ relates to the Dynkin operator. Recall that the space of linear endomorphisms of a graded connected Hopf algebra $H$ is equipped with the structure of an associative algebra by the convolution product. In that framework, the Dynkin operator $D$ is defined as the convolution product of the antipode $S$ of the Hopf algebra (i.e.~the convolution inverse of the identity $id$ of $H$) with the graduation operator $Y$ (acting as the multiplication by $n$ on the degree $n$ component of $H$). When acting on the double Spitzer algebra $\mathcal C$ associated to the free Rota--Baxter algebra over a generator $b$, the Dynkin operator satisfies the identity
$$
	D(R^{[n}(b))=D(R(\cdots R(R(b)b)\cdots )b)= b\rhd_\theta (b\rhd_\theta (\cdots  (b\rhd_\theta b)\cdots )).
$$
We refer to \cite{EGBP,EMP}, where this identity was proven, for general aspects including further references on Hopf algebras and Dynkin operators.}}
\end{remark}  

The following statement is of crucial importance for the understanding of the combinatorics of time-ordered products (the link between time-ordered products and Rota--Baxter algebras will be detailed in the next section). We state it therefore as a theorem.

\begin{theorem} \label{thm:keyth}
The image of a symmetric product $b_1 \cdots b_n$ in $GL$, $b_i \in \mathcal D$, by the previous canonical embedding $\iota$ is the following element, which we call the {\it{pre-time-ordered product}} of the $b_i$ in $RB$:
 \begin{equation}
 \label{keyeq}
 	\iota (b_1\cdots b_n) = \sum\limits_{\sigma \in S_n} R^{[\sigma}(b_1,\ldots ,b_n),
 \end{equation}
where: 
$$
	R^{[\sigma}(b_1,\ldots ,b_n) := R^{[n}(b_{\sigma(1)},\ldots , b_{\sigma(n)}) 
						     = R\bigl(\cdots R( R(b_{\sigma(1)})b_{\sigma(2)})\cdots \bigr)b_{\sigma(n)} .
$$
Observe that $R^{[\id}(b,\ldots ,b) = R^{[n}(b)$. 
\end{theorem}

\begin{proof}
The reason for calling these elements (pre-)time-ordered products follows from the fact that its image under the Rota--Baxter map $R$, e.g., when $R=I$ is the Riemann integral operator, is the usual time-ordered product. Further below we will dwell on this. 

Notice first that polarization applies. Defining $b:=\lambda_1 b_1 + \cdots + \lambda_n b_n$, then $\iota(b_1 \cdots b_n)$ is the coefficient of $\lambda_1 \cdots \lambda_n$ in the expansion of $\iota(b^n)$ as a polynomial in the $\lambda_i$. A similar observation holds for the right hand side. Hence, we are left with showing that: 
\begin{equation}
\label{recurs}
	\iota (b^n) = n! R^{[\id}(b,\ldots ,b)
\end{equation}
for an arbitrary element $b$ in the linear span of $\mathcal D$.

The proof is then by induction on $n$. For $n=2$ we find:
\begin{eqnarray*}
	\iota(b^2) = \iota(b*b - b \curvearrowleft b) 	&=& b *_\theta b - b \lhd_\theta b\\ 
									&=& b *_\theta b + b \rhd_\theta b =2R(b)b. 
\end{eqnarray*}

Assume that identity (\ref{recurs}) is true up to order $n-1$. Then:
\begin{eqnarray*}
	\iota(b^n) 	&=& \iota\big(b^{n-1} \ast b  - (n-1)b^{n-2} (b \curvearrowleft b)\big)\\
			&=& \iota\big(b^{n-1} \ast b - (n-1)b^{n-2} \ast (b \curvearrowleft b) 
			+  (n-1)(n-2)b^{n-3} (b \curvearrowleft (b \curvearrowleft b))\big)\\
			&=& \iota\big( \sum_{k=1}^{n-1} (-1)^{k-1} \frac{(n-1)!}{(n-k)!} b^{n-k} \ast b^{\curvearrowleft k}\big),
\end{eqnarray*}
where $b^{\curvearrowleft 1}:=b$, and $b^{\curvearrowleft k}:= (b \curvearrowleft b^{\curvearrowleft (k-1)})$. Hence, with $\iota(b \curvearrowleft b^{\curvearrowleft (k-1)})) = - \iota(b^{\curvearrowleft (k-1)}) \rhd_\theta b$, this yields:
 \begin{equation}
 \label{sol}
	\frac{\iota(b^n)}{(n-1)!} =  \sum_{k=1}^{n-1} R^{[(n-k)}(b,\ldots ,b) \ast_\theta b^{\rhd_\theta k},
 \end{equation}
where $b^{\rhd_\theta k} := (b^{\rhd_\theta (k-1)} \rhd_\theta b)$. The right hand side of (\ref{sol}) is exactly the convolution product $\id \star D = \id \star (S \star Y)=Y$ applied to $R^{[n}(b)$ in the context of the double Spitzer (Hopf) algebra $\mathcal{C}$ \cite{EMP}. Therefore we obtain $\iota(b^n) = (n-1)!nR^{[n}(b)$. 
\end{proof}

We obtain for $\iota(b^n)$ the expansion:
 \begin{equation}
 \label{keyeq2}
	n!R^{[n}(b) = \sum\limits_{{s_1+\cdots + s_k=n}\atop {s_i>0}} c(s_1,\cdots , s_k)\
				b^{\rhd_\theta s_1} \ast_\theta  \cdots \ast_\theta  b^{\rhd_\theta s_k},  
\end{equation}
where the numbers $c(s_1,\cdots , s_k)$ arise from choosing $l_1=\cdots = l_n=b$ in Proposition \ref{prop:Envel}. Since the elements $b^{\rhd_\theta l}$, $l > 0$, are algebraically independent in the double Spitzer algebra, we can deduce from the results in \cite{EMP} on the Dynkin map and its inverse that:
 \begin{equation}
 \label{keyeq3}
 	c(s_1,\ldots , s_k) =  \frac{n!}{(\Pi_{j=1}^k (s_1+\cdots + s_k))} 
\end{equation}
  
Eventually, after polarization of (\ref{recurs}) we obtain the noncommutative generalization of the well-known Bohnenblust--Spitzer identity \cite{EMP}.
\begin{corollary}
\label{cor:Bsp}
 \begin{equation}
 \label{BSp1}
	\sum_{\sigma \in S_n} R^{[\sigma}(b_1,\ldots, b_n) 
	= \sum\limits_{P_1,\ldots ,P_k} b_{P_1}\ast \cdots \ast b_{P_k},  
 \end{equation}
where $P:=\{P_1,\ldots ,P_k\}$ runs over the partitions of $[n]=\{1,\ldots, n\}$ ($P_1\coprod \cdots \coprod P_k=[n]$), such that $\sup(P_i)>j$  $\forall j \in P_k$, $k<i$. Besides, for $P_i=\{p_1^i,\ldots ,p_h^i\}$, where the $p_j^i$ are in the natural order, we set:
 \begin{equation}
 \label{BSpPart}
	b_{P_i}:=\sum_{\sigma\in S_{h-1}}
	( \cdots (b_{p_{h}^i} \rhd_\theta b_{p_{\sigma(h-1)}^i})\cdots \rhd_\theta b_{p_{\sigma(2)}^i} )\rhd_\theta b_{p_{\sigma(1)}^i}  .
 \end{equation}
\end{corollary}

\noindent See also \cite{NovThib}, where the identity was proved using techniques of (colored) free quasi-symmetric functions.   

Looking at  Corollary \ref{cor:Bsp}, it seems to be a natural question to ask whether one can rewrite the right hand side of (\ref{BSp1}) as a sum over permutations in $S_n$ instead of partitions. Indeed, in \cite{EMP} this was done using the notion of left-to-right maxima, or records, of permutations. 

Recall that an element $\sigma(i)$ in a permutation $\sigma=(\sigma(1),\ldots,\sigma(n)) \in S_n$ is a record if it is strictly larger than all elements to its left. The position $i$ is the position of the record. In a canonical cycle decomposition of a permutation, we put the maximal element in each cycle first.  We denote by $m_j$ the number of elements in cycle $c_j$. The cycles are then ordered from left to right in increasing order of the cycle maxima, i.e., the leftmost element in each cycle. We enumerate the elements inside each cycle from $0$ to $m_j-1$, $c_j=({j_0},{j_1},\ldots,{j_{m_j-1}})$, where ${j_0}> j_l$, ${1} \le l \le {m_j-1}$. Recall the bijection $q: S_n \to S_n$, which associates with each permutation a unique canonical cycle decomposition, $q(\sigma) = c_1 \cdots c_k$. The cycles are obtained by introducing parentheses at each record position. We denote by $|q(\sigma)|=k$ the number of cycles. For example, the permutation $\sigma=(32541687) \in S_8$ is mapped to the canonical cycle decomposition $q(\sigma)=(32)(541)(6)(87)$. 

With this notation we obtain from (\ref{BSp1}) the identity:
\begin{equation}
\label{ncBSp}
	\sum_{\sigma \in S_n} R\bigl(\cdots R( R(b_{\sigma(1)})b_{\sigma(2)})\cdots \bigr)b_{\sigma(n)} 
	=	\sum_{\sigma \in S_n}  \mathrm{lr}^{*_{\theta}}_{\sigma}(b_1,\ldots,b_n).
\end{equation}
The map $\mathrm{lr}^{*_{\theta}}_{\sigma}$ is defined using $q(\sigma)=c_1 \cdots c_k$ by:
$$
	\mathrm{lr}^{*_{\theta}}_{\sigma}(b_1,\ldots,b_n) 
	:= \sideset{}{^{*_{\theta}}}\prod_{j=1}^{k}\Big( \prod_{i=1}^{m_j-1}\mathrm{r}_{\rhd_\theta b_{j_i}}(b_{j_0})\Big)
	= \sideset{}{^{*_{\theta}}}\prod_{j=1}^{k}\Big( (b_{j_0}\rhd_\theta  b_{j_1}) \cdots \rhd_\theta b_{j_{m_j-1}}\Big),
$$   
with the right multiplication operator $\mathrm{r}_{\rhd_\theta x}(y):=y \rhd_\theta x $. As an example consider $\sigma:=(43512) \in S_5$, with $q(43512)=(43)(512)$, we find:
$$
	\mathrm{lr}^{*_{\theta}}_{\sigma}(b_1,\ldots,b_5) 
	= (b_4 \rhd_\theta b_3)*_{\theta}((b_5 \rhd_\theta b_1) \rhd_\theta b_2).
$$  
Note that the sum (\ref{BSpPart}) corresponds to the $(m_i-1)!$ ordered cycles $(i_0,i_{\sigma(1)},\ldots, i_{\sigma(m_i-1)})$ where ${i_0}> i_l$, ${1} \le l \le {m_i-1}$. 

Observe that in the case of a commutative Rota--Baxter algebra we find that $\mathrm{r}_{\rhd_\theta x}(y)= \theta xy = \theta yx$, and that (\ref{ncBSp}) reduces to the classical expression:
$$
	\sum_{\sigma \in S_n} R\bigl(\cdots R( R(b_{\sigma(1)})b_{\sigma(2)})\cdots \bigr)b_{\sigma(n)} 
			= \sum_{P_1, \ldots,P_k} \theta^{n -|k|}
				\sideset{}{^{*_{\theta}}}\prod_{i =1}^k(|P_i|-1)!  \prod_{j \in P_i}b_j,
$$   
where $P:=\{P_1, \ldots,P_k\}$ runs over set partitions of $[n]=\{1,\ldots, n\}$, i.e., $P_1\coprod \cdots \coprod P_k=[n]$. The number of elements in $P_i$ is denoted $|P_i|$. The link to identity (\ref{ncBSp}) follows from the fact that $(|P_i|-1)!$ distinct cycles $(i_{0},i_{\sigma(1)},\ldots,i_{\sigma({|P_i|-1})})$, $\sigma \in S_{m_j-1}$, are mapped to the same block $P_i$.


\section{Time-ordered exponential and Bohnenblust--Spitzer identitiy}
\label{sect:ToBSp}

Now, we explore time-ordered products from the point of view of Rota--Baxter algebra. Note that \cite{EMP} was in parts motivated by Lam's paper \cite{lam1998} on a decomposition of time-ordered products and path-ordered exponentials.

Following \cite{lam1998} the notion of time-ordering, represented in terms of the so-called time-ordering operator $\mathsf{T}$, plays a central role in applied mathematics as well as theoretical physics. Recall that it is defined for functions $U_1,\ldots, U_n$ at distinct times $s_1,\ldots, s_n$ as follows:
\begin{equation}
\label{TOrd}
	{\mathsf{T}}\big[U_1(s_1) \cdots U_n(s_n)\big] := 
	\sum_{\sigma \in S_n} \Theta^\sigma(s_1, \ldots s_n)\ U_{\sigma(1)}(s_{\sigma(1)}) \cdots U_{\sigma(n)}(s_{\sigma(n)}),
\end{equation}
where $ \Theta^\sigma:=\Theta^\sigma(s_1, \ldots , s_n)$ involves the Heaviside step function $\Theta^\sigma(s_1, \ldots, s_n):= \Pi_{i=1}^{n-1} \Theta(s_{\sigma(i)} - s_{\sigma(i+1)}).$ The functions $U_i$ may be matrix or operator valued. Observe that (\ref{TOrd}) implies that for any permutation $\sigma \in S_n$: 
$$
	{\mathsf{T}}\big[U_1(s_1) \cdots U_n(s_n)\big]={\mathsf{T}}\big[U_{\sigma(1)}(s_{\sigma(1)}) \cdots U_{\sigma(n)}(s_{\sigma(n)})\big].
$$ 
The benefit of the notion of time-ordering is best described in the context of linear initial value problems (IVP), which consist of a first order linear differential equation
\begin{equation}
\label{eq:ivp}
	\dot{Y}(t)= U(t)Y(t),
\end{equation}
together with the initial value $Y(0) = Y_0$. Corresponding to (\ref{eq:ivp}) we find the linear integral fixpoint equation:
\begin{equation}
\label{eq:fixpoint} 
	Y(t)=Y_0 + \int_0^tU(s)Y(s)ds,
\end{equation}
from which the formal solution of (\ref{eq:ivp}) is derived in terms of the time-ordered exponential function:
\begin{equation}
\label{eq:Torder} 
	{\mathsf{T}}\!\exp\Bigl(\int_0^t U(s)ds\Bigr)Y_0 := Y_0\un + 
			\sum_{n>0} \frac{1}{n!} \int_{[0,t]^n} {\mathsf{T}}\big[U(t_1) \cdots U(t_n)\big]dt_1 \cdots dt_n Y_0.
\end{equation}
The first few terms are:
\begin{equation}
\label{eq:TorderExplicit} 
	 Y(t)=  \Big(\un + \int_0^t \!\!\!\!U(t_1) dt_1 \
	   +  \int_0^t \!\!\!\!U(t_1)\!\!\int_0^{t_1} \!\!\!\! U(t_2) dt_2dt_1\
	   +  \int_0^t \!\!\!\!U(t_1)\!\!\int_0^{t_1} \!\!\!\! U(t_2) \!\! \int_0^{t_2} \!\!\!\!U(t_3) dt_3dt_2dt_1\ + \cdots\Big)Y_0 .
\end{equation}
This follows directly from the obvious identity:
\begin{eqnarray*}
	\int_{[0,t]^n} {\mathsf{T}}\big[U(t_1) \cdots U(t_n)\big] dt_1 \cdots dt_n &=& 
	\int_{[0,t]^n} \sum_{\sigma \in S_n} \Theta^\sigma\, U(t_{\sigma(1)}) \cdots U(t_{\sigma(n)})dt_1 \cdots dt_n\\
	&=& n! \idotsint\limits_{0\le t_1 \le \cdots \le t_n \le t }U(t_1) \cdots U(t_n)dt_1 \cdots dt_n,
\end{eqnarray*}
which is a special case of the more general relation:
\begin{eqnarray}
	\int_{[0,t]^n} {\mathsf{T}}\big[U_1(t_1) \cdots U_n(t_n)\big] dt_1 \cdots dt_n &=& 
	\int_{[0,t]^n}\sum_{\sigma \in S_n} 
			\Theta^\sigma U_{\sigma(1)}(t_{\sigma(1)}) \cdots U_{\sigma(n)}(t_{\sigma(n)})dt_1 \cdots dt_n \nonumber\\
	&=& \sum_{\sigma \in S_n}\;\;  \idotsint\limits_{0\le t_{\sigma(1)} \le \cdots \le t_{\sigma(n)} \le t }
			U_{\sigma(1)}(t_{\sigma(1)}) \cdots U_{\sigma(n)}(t_{\sigma(n)})dt_1 \label{BSpZero1}
	\cdots dt_n,
\end{eqnarray}

Hence we may characterize the time-ordered exponential ${\mathsf{T}}\!\exp\bigl(\int_0^t U(s)ds\bigr)$ as the completely symmetrization of the n-fold iterated integral of the functions $U_i$. Note however that contrary to the commutative, e.g., scalar case, this does not coincide with the n-fold product of integrals. 

With motivations coming from questions arising from theoretical physics, Lam hinted in \cite{lam1998}  at a non-trivial rewriting problem of these fully symmetrized n-fold iterated integrals essentially in terms of products of single integrals and Lie brackets. As an example we state the identity:
\begin{equation}
\label{eq:pL1} 
	\int_{[0,t]^2} {\mathsf{T}}\big[U_1(t_1) U_2(t_2)\big] dt_1 dt_2 =
	\int_{0}^{t} U_1(t_1)dt_1 \int_{0}^{t}  U_2(t_2) dt_2 + \int_{0}^{t}  ad_{\int_{0}^{t_1}  U_2(t_2)dt_2}\big( U_1(t_1)\big) dt_1
\end{equation}
which follows from integration by parts. Note that the Lie bracket term on the right hand side of (\ref{eq:pL1}) defines a left pre-Lie product, which naturally enters the description of solutions of IVPs. 

As it turns out Lam's problem amounts to finding a general noncommutative analog of the classical Bohnenblust--Spitzer identity, i.e., identity (\ref{BSp1}), respectively (\ref{ncBSp}). We therefore find from (\ref{keyeq2}) for weight $\theta=0$ and denoting $\rhd_0=\rhd$:

\begin{proposition}
\begin{equation}
\label{eq:newTordExp} 
	{\mathsf{T}}\!\exp\Bigl(\int_0^t U(s)ds\Bigr)
	\!=\!  \un + \sum_{n>0} \frac{1}{n!}\sum\limits_{{k_1+\cdots + k_l=n}\atop {k_i>0}}  \frac{n!}{(\Pi_{j=1}^l (k_1+\cdots + k_j))} 
				\int_0^t U^{\rhd k_1}(s_1)ds_1  \cdots\!\!  \int_0^t U^{\rhd k_l}(s_l)ds_l,
\end{equation}
where $U^{\rhd 1}(t)=U(t)$, and: 
$$
	U^{\rhd k+1}(t)=(U^{\rhd k}(t))\rhd U(t) = ad_{\int_0^tU^{\rhd k}(s)ds}(U(t)).
$$
\end{proposition}

And from the more general point of view of Rota--Baxter algebras of arbitrary weight, we may therefore define the algebraic analog time-ordering  by:
$$
	{\mathsf{T}}R(b_1,\ldots, b_n) :=  \sum_{\sigma \in S_n} R^{[\sigma]}(b_1,\ldots, b_n). 
$$


\section{On Magnus and continuous BCH expansions}
\label{sect:Magnus}

In the present section, we analyse briefly how the pre-Lie point of view on Rota--Baxter algebras impacts some of the classical results in the theory of linear differential equations. This analysis relies largely on (and complements) recent works by F.~Chapoton on the links between free Lie and free pre-Lie algebras from the point of view of Lie idempotents \cite{Chap3}, by J.-C.~Novelli and J.-Y.~Thibon on the colored free quasi-symmetric functions interpretation of the noncommutative Bohnenblust--Spitzer formula \cite{NovThib}, and on joint works of the authors with F.~Chapoton \cite{ChapPat} and D.~Manchon \cite{EM}.

Recall that the Baker--Campbell--Hausdorff formula computes the logarithm of the solution of a linear differential equation in an algebra of time-dependent operators. Abstractly, this amounts to compute the logarithm of a time-ordered exponential ${\mathsf{T}}\!\exp$ or, equivalently, the logarithm of the series $\un+ \sum_n R(R^{[n}(x)) = \un+ \sum_n R^{[n]}(x)$ --in a weight zero Rota--Baxter algebra. The discrete case, that is, the computation of the logarithm of a product of exponentials in a noncommutative algebra, is a particular case of the continuous one (since a product of exponentials can always be obtained by solving a differential equation whose infinitesimal generator is a time-dependent step function, see e.g.~\cite{MielPleb} for details). There are several combinatorial solutions to these problems, surveyed for example in Reutenauer's monograph \cite{reutenauer}. We will be mainly interested here in the Mielnik--Pleba\'nski--Strichartz solution \cite{MielPleb,Strichartz} in terms of what is now called the series of Solomon idempotents (also referred to as eulerian or canonical idempotents in the literature) --see e.g.~\cite{ChapPat} for an account on Solomon idempotents in the context of enveloping algebras of pre-Lie algebras.

For a permutation $\sigma \in S_n$, let us write 
$$
	U_\sigma := \idotsint\limits_{0\le t_1 \le \cdots \le t_n \le t }U(t_{\sigma(1)}) \cdots U(t_{\sigma(n)})dt_1 \cdots dt_n.
$$
and $U_n:=U_{1_n}$, where $1_n \in S_n$ stands for the identity permutation. Recall that a permutation $\sigma$ has a descent in position $i$ if and only if $\sigma(i)>\sigma(i+1)$. The descent set of $\sigma $, denoted $Desc(\sigma)$, is the set of all those $i$'s, such that $\sigma$ has a descent in position $i$. Elementary calculations show that $U_nU_m=\sum_{(\sigma)} U_{\sigma}$, where the sum runs over all permutations $\sigma$ with at most one descent, in position $n$. More generally, (see e.g.~\cite{AG2,brouder} for details on such computations):
$$
	U_{n_1} \cdots U_{n_m} =\sum\limits_{\sigma \in S_{\eta_m} \atop Desc(\sigma)\subset \{\eta_1,\eta_2,\ldots ,\eta_m\}} U_\sigma,
$$
where $\eta_i:=n_1+\cdots + n_i$. The Mielnik--Plenba\`nski--Strichartz formula follows:
\begin{eqnarray*}
	 \log \big({\mathsf{T}}\!\exp(\int_0^t U(x)dx)\big)
	 &=&\sum_n\sum_{S\subset[n-1]}\frac{(-1)^{|S|}}{n}\sum_{\sigma\in S_n \atop Desc(\sigma)\subset S}U_\sigma \\
	 &=&\sum_n\sum_{S\subset[n-1]}\frac{(-1)^{|S|}}{n}{n-1\choose  |S|}^{-1}\sum_{\sigma\in S_n \atop Desc(\sigma)=S}U_\sigma.
\end{eqnarray*}
The last identity follows by a M\"obius inversion argument in the integer partitions lattice. See \cite{reutenauer} for a proof.
The elements: 
$$
	\sum_{S\subset[n-1]}\frac{(-1)^{|S|}}{n}{n-1\choose  |S|}^{-1}\sum_{\sigma\in S_n \atop Desc(\sigma)=S}\sigma
$$ 
in the symmetric group algebras are so-called Lie idempotents (Solomon's idempotents), i.e., they are idempotent and, when viewed as acting on the tensor algebra, map surjectively to the free Lie algebra~\cite{reutenauer}.

The same formula holds \it mutatis mutandis \rm for an arbitrary weight zero Rota--Baxter algebra. For subtler reasons, an analogous formula also holds in Rota--Baxter algebras of arbitrary weight. This is a consequence of the arguments in \cite{EGBP,EMP}, that show that the Spitzer algebra carries the structure of a cocommutative Hopf algebra, so that the arguments in \cite{pat} apply. In particular, recall from \cite{EGBP,EMP} that $\un+\sum_nR^{[n]}(x)$ is a group-like element in $\mathcal S$, and for the graded components $L_n(x)$ of $\log (\un+\sum_nR^{[n]}(x))$ (that is, $\sum_n L_n(x):=\log (\un+\sum_nR^{[n]}(x))$) the following formula holds:
\begin{equation}
\label{Hopflog}
	L_n(x)=\sum_{1\leq k\leq n}(-1)^{k-1}\Big(\sum_P\frac{I_P}{k}\Big)(R^{[n]}(x)),
\end{equation}
where the last sum runs over integer partitions $P=(p_1,\ldots,p_k)$ of $n = p_1+\ldots +p_k$, and $I_P:=I_{p_1}\star \cdots \star I_{p_k}$, where $I_n$ is the projection on the degree $n$ component of the Spitzer algebra. The product $\star$ stands for the convolution product of endomorphisms in the Spitzer algebra (viewed as a Hopf algebra). We refer the reader to \cite{pat} for details.

These results also hold in the double Spitzer algebra $\mathcal C$ (which is isomorphic as a Hopf algebra to the Spitzer algebra $\mathcal S$). In particular, equation~(\ref{Hopflog}) holds in $\mathcal C$ when the $R^{[n]}(x)$ are replaced by the $R^{[n}(x)$ (that is, the right hand side of the equation computes the degree $n$ component of $\log^{\ast_\theta} (\un+\sum_nR^{[n}(x))$, where the notation $\log^{\ast_\theta}$ means that the logarithm is computed with respect to the double product $\ast_\theta$ (\ref{def:RBdouble}). When the Rota--Baxter map is the Riemann integral and $x=U=U(t)$ is a time-dependent operator, notice that $\log^{\ast_\theta} (\un+\sum_nR^{[n}(U))$ is nothing but the derivative of the logarithm of the solution of the linear differential equation $\dot{X}=XU$.

The Magnus expansion on the other hand provides another computation of this derivative. Let us state directly the identity in terms of Rota--Baxter algebras.

\begin{theorem} {\rm{\cite{EM}}} \label{thm:pLMag}
Let $A$ be a Rota--Baxter algebra of weight $\theta \in k$. Assume that $x=x(a)$ is a solution of  $x =a+R(x)  a$. The element $\Omega'(a) \in P_A$, such that $\un + x =\exp^{*_\theta}\big(\Omega'(a)\big)$, satisfies:
\begin{equation}
\label{pLMagnus}
	\Omega'(a) = \frac{-\ell_{\Omega'(a) \rhd_\theta}}{e^{-\ell_{\Omega'(a) \rhd_\theta}} - 1}(a) 
	= a + \sum_{n > 0} (-1)^n\frac{B_n}{n!} \ell^{n}_{\Omega'(a) \rhd_\theta}(a).
\end{equation}
\end{theorem}  

The left multiplication operator $\ell_{u \rhd_{\theta}}$ is defined by $\ell_{u \rhd_{\theta}}(v):=u \rhd_{\theta} v$, and the $B_i$ are the Bernoulli numbers. The first few terms of $\Omega'(a)$ are:  
\begin{equation*}
	\Omega'(a) = a + \frac 12 a\rhd_\theta a - \frac 14 (a \rhd_\theta  a) \rhd_\theta a 
				- \frac 1{12} a \rhd_\theta (a \rhd_\theta a) + \cdots.
\end{equation*}

The equivalence between these two solutions of the continuous Baker--Campbell--Hausdorff problem is all but obvious from a combinatorial point of view. In fact, the classical proofs of the Mielnik--Pleba\'nski--Strichartz and Magnus formulae usually rely on largely different techniques, i.e., the combinatorics of iterated integrals and symmetric groups versus integral identities such as Duhamel's formula, respectively. In our opinion the perspective offered by the Grossman--Larson algebra and its canonical map to the free Rota--Baxter algebra may enlighten this relationship.

Recall first from \cite{ChapPat} that, in the $GL$ algebra over a generator $a$ (the enveloping algebra over the free Lie algebra over one generator), the following property holds: for $\Omega':=\log^\ast(\exp(a))$, where the logarithm is computed with respect to the Grossman--Larson (associative) product, and the exponential with respect to the commutative product, we have:
\begin{equation}
\label{GLeq}
	\Omega' = a\curvearrowleft \Big(\frac{\Omega'}{\exp(\Omega')-1}\Big).
\end{equation}

On the other hand, $\exp(a)$ is a group-like element in $GL$ (this can be proved directly, but also follows from the fact that its image in the free Rota--Baxter algebra belongs to the double Spitzer algebra, and is the group-like series $\un+ \sum_nR^{[n}(a)$). In particular, equation~(\ref{Hopflog}) holds in $GL$, when the $R^{[n]}(a)$ are replaced by the $a^n$ (that is, the right hand side of the equation computes the degree $n$ component of $\log^{\ast} (\exp(a))$). 

The existence of two different formulae for the computation of $\log^{\ast} (\exp(a))$ in $GL$ can be understood as a formal property of pre-Lie algebras and their enveloping algebras. In fact, these formulae hold not just in $GL$ but in any enveloping algebra of a pre-Lie algebra.

\smallskip 

As a conclusion we state that Theorem \ref{thm:keyth} (and other results in the present article following from this Theorem) has then several consequences. First, each of these formulae in $GL$ implies the corresponding formula in the free Rota--Baxter algebra (Mielnik--Pleba\'nski--Strichartz's as well as Magnus'). Notice that the converse statement is also true since the canonical map from the $GL$ algebra to the free Rota--Baxter algebra is an embedding, so that these formulae stated in Rota--Baxter algebras are equivalent to the same formulas stated in enveloping algebras of pre-Lie algebras. Second, in a more subtle way, it results that the equivalence of the Mielnik--Pleba\'nski--Strichartz and Magnus solutions of the continuous Baker--Campbell--Hausdorff problem that could hardly be understood in the classical approach to these solutions can now be understood on structural grounds, i.e., as a consequence of the formal properties of pre-Lie algebras.


\section{Products of time-ordered exponentials}
\label{sect:prodToExp}

Before concluding this article, we would like to indicate the surprising mathematical transparency following from lifting the calculation of the solution of a linear differential equation into $GL$. We will do so by calculating the product of $\exp(x)$ and $\exp(y)$ in $GL$.

\begin{proposition}
For $x,y \in \mathcal{D}$, the product of $\exp(x)$ and $\exp(y)$ in $(GL,*)$ is given by:
$$
	\exp(x) * \exp(y) = \exp\big(y + \{x; \exp(y)\}\big).
$$
\end{proposition}

Note that the exponential is with respect to the commutative product in $GL$, whereas the Grossman--Larson (associative) product $*$ is noncommutative. Definition \ref{def:forest} implies:
$$
	\exp(x) * \exp(y) = (\exp(x)  \curvearrowleft  \exp(y))\exp(y),
$$ 
since $\exp(y)$ is group-like in $GL$. Using (\ref{forestaction}), we see that:
$$
	\exp(x)  \curvearrowleft  \exp(y) = \sum_{n\ge 0} \frac{x^n \curvearrowleft  \exp(y)}{n!} 
	= \sum_{n\ge 0} \frac{(x\curvearrowleft  \exp(y))^n}{n!} = \exp(\{x;  \exp(y)\}).
$$
This eventually yields for the product of two time-ordered exponentials:  
\begin{equation}
\label{prodTordExp}
	\exp(x) * \exp(y) = (\exp(x)  \curvearrowleft  \exp(y))\exp(y) =  \exp(y + \{x;  \exp(y)\}).
\end{equation}

The argument of the exponential on the right hand side is of particular interest. Remember, we are essentially calculating the product of solutions of two linear differential equations. Indeed, using (\ref{GLeq}) we deduce from $\exp(y)=\exp^*(\log^*(\exp(y)))$ that 
$$
	\exp(y)=\exp^*(\Omega'(y)).
$$
Such that $\{x;  \exp(y)\} = \{x; \exp^*(\Omega'(y))\}$, which therefore gives:
$$
	\{x; \exp^*(\Omega'(y))\} = \sum_{n\ge 0} \frac{\{x; (\Omega'(y))^{*n}\}}{n!} 
					     = \sum_{n\ge 0} \frac{1}{n!} x \curvearrowleft (\Omega'(y))^{*n}.
$$
This reduces to:
$$
	\{x; \exp^*(\Omega'(y))\} = x + \sum_{n > 0} \frac{1}{n!} (\cdots (( x\curvearrowleft \Omega'(y))\curvearrowleft \Omega'(y)) 
			   		\cdots \curvearrowleft \Omega'(y)) = e^{r_{\curvearrowleft\Omega'(a)}}x,
$$
where, as before, $r_{\curvearrowleft a}(b):=b \curvearrowleft a$. From this we deduce that:
$$
	y + \{x; \exp^*(\Omega'(y))\} = y +  e^{r_{\curvearrowleft\Omega'(y)}}x.
$$
The right hand side of the last equality defines a product $\#$:
\begin{equation*}
	a\# b=a+e^{r_{\curvearrowleft  \Omega'(a)}}b.
\end{equation*}
It appeared already in \cite{AG}, and provides a way to understand the relationship between the pre-Lie Magnus expansion (\ref{GLeq}) and the discrete Baker--Campbell--Hausdorff formula:
\begin{equation*}
	BCH(x,y) := x+y +\frac 12[x,y] + \frac 1{12}([x,[x,y]]+[y,[y,x]]) + \cdots ).
\end{equation*} 
We refer the reader to Manchon's expository article \cite{Manchon} for more details.



\end{document}